\documentclass[12pt,leqno]{article}
\usepackage{amsfonts,amsthm,amsmath}

\theoremstyle{plain}
\newtheorem{thm}{Theorem}[section]

\newtheorem{lem}{Lemma}[section]

\newtheorem{conj}{Conjecture}[section]

\theoremstyle{definition}
\newtheorem{df}{Definition}[section]
\newtheorem{rem}{Remark}[section]

\newcommand{\FF}{\mathbb{F}}
\newcommand{\ZZ}{\mathbb{Z}}

\newcommand{\CC}{\mathbb{C}}

\newcommand{\0}{\mathbf{0}}

\newcommand{\QQ}{\mathbb{Q}}

\newcommand{\C}{\mathbb{C}}

\DeclareMathOperator{\wt}{wt}

\begin{document}

\title{{
On Eisenstein polynomials and zeta polynomials II
\footnote{This work was supported by JSPS KAKENHI (18K03217, 17K05164).}
}
}

\author{
Tsuyoshi Miezaki
\thanks{Faculty of Education, University of the Ryukyus, Okinawa  
903-0213, Japan, 
miezaki@edu.u-ryukyu.ac.jp
(Corresponding author)
}
and
Manabu Oura
\thanks{Graduate School of Natural Science and Technology, 
Kanazawa University,  
Ishikawa 920-1192, Japan, 
oura@se.kanazawa-u.ac.jp,
Telephone: +81-76-264-5635, Fax: +81-76-264-6065 
}
}

\date{}
\maketitle

\begin{abstract}
Eisenstein polynomials, which were defined by 
the second author, are analogues of the 
concept of an Eisenstein series. 
The second author conjectured that there exist some analogous properties 
between Eisenstein series and Eisenstein polynomials.
In the previous paper, the first author 
provided new analogous properties of Eisenstein polynomials and 
zeta polynomials for the Type II case. 
In this paper, 
the analogous properties of Eisenstein polynomials and 
zeta polynomials are shown to also hold for the Type I, Type III, and Type IV cases. 
These properties are finite analogies of 
certain properties of Eisenstein series. 
\end{abstract}

{\small
\vspace{-0.15in}

\quad
}


\section{Introduction}
In the present paper, 
we discuss some analogies between 
Eisenstein series, Eisenstein polynomials, and zeta polynomials. 
This paper is a sequel to the 
paper \cite{miezaki}, which we refer to in order  
to explain our results here. 
A linear code $C$ of length $n$ is a linear subspace of $\FF_{q}^{n}$. Then, the dual $C^{\perp}$ of a linear code $C$ is defined as follows: $C^{\perp}=\{ \bold{y}\in \FF_{q}^{n}\ | \ ( \bold{x},\bold{y}) =\0\ \mbox{ for all }\bold{x}\in C\}$. A linear code $C$ is called self-dual if $C=C^{\perp}$. The weight $\wt(\bold{x})$ is the number of its nonzero components. 
The weight enumerator of a code $C$ is
\begin{align*}
w_C(x, y)&=\sum_{\bold{u}\in C}x^{n-\wt(\bold{u})}y^{\wt(\bold{u})}=x^{n}+\sum_{i=1}^{n}A_{i}x^{n-i}y^{i}, 
\end{align*}
where $A_{i}$ is the number of codewords of weight $i$. 
In this paper, we consider the following self-dual codes~\cite{CS}: 
\begin{tabbing}
Type I: A code is defined over $\FF_{2}^{n}$ with all weights divisible by $2$,\\
Type II: A code is defined over $\FF_{2}^{n}$ with all weights divisible by $4$,\\
Type III: A code is defined over $\FF_{3}^{n}$ with all weights divisible by $3$,\\
Type IV: A code is defined over $\FF_{4}^{n}$ with all weights divisible by $2$. \
\end{tabbing}
For the detailed expression of codes, see \cite{{CS},{E}}. 
For Types $\mbox{I}$ to $\mbox{IV}$, it is well known that the weight enumerator $w_C(x, y)$ is in the ring of invariants $\C[f, g]$ \cite{CS}, where 
\begin{tabbing}
Type I : $f=x^2+y^2$, $g=x^2y^2(x^2-y^2)^2$, \\
Type II: $f=x^8+14x^4y^4+y^8$, $g=x^4y^4(x^4-y^4)^4$, \\
Type III: $f=x^4+8xy^3$, $g=y^3(x^3-y^3)^3$, \\
Type IV: $f=x^2+3y^2$, $g=y^2(x^2-y^2)^2$. 
\end{tabbing}
Let 
\begin{tabbing}
Type I: $G_{\rm{I}}=
\displaystyle\left\langle 
\frac{1}{\sqrt{2}}
\begin{pmatrix}
1&1\\
1&-1
\end{pmatrix}, 
\begin{pmatrix}
1&0\\
0&-1
\end{pmatrix}
\right\rangle$, \\
Type II: $G_{\rm{II}}=
\displaystyle\left\langle 
\frac{1}{\sqrt{2}}
\begin{pmatrix}
1&1\\
1&-1
\end{pmatrix}, 
\begin{pmatrix}
1&0\\
0&\sqrt{-1}
\end{pmatrix}
\right\rangle$, \\
Type III: $G_{\rm{III}}=
\displaystyle\left\langle 
\frac{1}{\sqrt{3}}
\begin{pmatrix}
1&2\\
1&-1
\end{pmatrix}, 
\begin{pmatrix}
1&0\\
0&e^{2\pi \sqrt{-1}/3}
\end{pmatrix}
\right\rangle$, \\
Type IV: $G_{\rm{IV}}=
\displaystyle\left\langle 
\frac{1}{2}
\begin{pmatrix}
1&3\\
1&-1
\end{pmatrix}, 
\begin{pmatrix}
1&0\\
0&-1
\end{pmatrix}
\right\rangle$. 
\end{tabbing}
It is known that 
for $\rm{X}\in\{\rm{I},\ldots,\mbox{IV}\}$, 
a weight enumerator of Type X codes 
is an invariant polynomial of the 
group $G_{\rm{X}}$, namely, for all $\sigma\in G_{\rm{X}}$, 
\[
w_C(\sigma(x,y))=w_C(x,y), 
\]
where $\sigma(x,y):=\sigma\, {}^t(x,y)$. 
We denote by $\CC[x,y]^{G_{\rm{X}}}$ the 
$G_{\rm{X}}$-invariant subring of $\CC[x,y]$.

Oura defined an
Eisenstein polynomial for Type II 
as
\[
\varphi_\ell^{G_{\rm{II}}}(x,y)
=\frac{1}{|G_{\rm{II}}|}\sum_{\sigma\in G_{\rm{II}}}(\sigma x)^\ell, 
\]
where for 
$
\sigma:=
\begin{pmatrix}
a&b\\
c&d
\end{pmatrix}, 
\sigma x:=ax+by
$ \cite{Oura1,Oura2}.
Here we define an
Eisenstein polynomial for Type X 
as follows: 
\[
\varphi_\ell^{G_{\rm{X}}}(x,y)
=\frac{1}{|G_{\rm{X}}|}\sum_{\sigma\in G_{\rm{X}}}(\sigma x)^\ell. 
\]
It is straightforward to show that the 
Eisenstein polynomial for Type X is in $\CC[x,y]^{G_{\rm{X}}}$. 

We next introduce an expression relating 
$G_{\rm{II}}$ and modular forms 
$M(\Gamma_1)$. For the detailed expression of modular forms, see \cite{{CS},{E},{F1},{F2},{K},{Kob}}. 
For $\CC[x,y]^{G_{\rm{X}}}$, 
we construct the elements of $\Gamma_1$ as follows: 
\begin{align*}
Th: \CC[x,y]^{G_{\rm{X}}} &\rightarrow M(\Gamma_1)\\
x&\mapsto f_{0}(\tau)=\sum_{b\in \ZZ, b\equiv {0}\pmod{2}}
\exp(\pi \sqrt{-1}\tau {b^2}/2),\\ 
y&\mapsto f_{1}(\tau)=\sum_{b\in \ZZ, b\equiv {1}\pmod{2}}
\exp(\pi \sqrt{-1}\tau {b^2}/2). 
\end{align*}
The map $Th$ is called the theta map. 

A typical example of 
$M(\Gamma_1)$
is the Eisenstein series for $\Gamma_1$, which is defined as follows: 
\[
\psi_k^{\Gamma_1}(\tau):=1-\frac{2k}{B_k}\sum_{n=1}^{\infty}\sigma_{k-1}(n)q^n, 
\]
where $B_k$ is the $k$-th Bernoulli number, 
$\sigma_{k-1}(n):=\sum_{d\mid n}d^{k-1}$, and $q=e^{2\pi i\tau}$. 
For the detailed expression of the Eisenstein series, 
see \cite{{CS},{F1},{F2},{K},{Kob}}. 

The elements of both $M(\Gamma_1)$ and 
$\CC[x,y]^{G_{\rm{X}}}$ are ``invariant functions" 
and the Eisenstein series 
and the Eisenstein polynomial are 
``average functions" of the groups. Therefore, 
these two objects are expected to have similar properties. 
Moreover, for $f \in \CC[x,y]^{G_{\rm{X}}}$, 
it is expected that $f$ and $Th(f)$ have similar properties. 
Table \ref{Tab:sum} shows a summary of the concepts 
that we have introduced so far. 
\begin{table}[thb]
\caption{Summary of our objects}
\label{Tab:sum}
\begin{center}
\begin{tabular}{c|c}
\noalign{\hrule height0.8pt}
$\Gamma_1$ & $G_{\rm{II}}$\\\hline
$M(\Gamma_1)$ & $\CC[x,y]^{G_{\rm{II}}}$\\\hline
Eisenstein series & Eisenstein polynomials\\\hline
$f$ & $Th(f)$\\
\noalign{\hrule height0.8pt}
\end{tabular}
\end{center}
\end{table}

For $\varphi_{\ell}^{G_{\rm{X}}}(x,y)\not\equiv 0$, 
we denote by $\widetilde{\varphi_{\ell}^{G_{\rm{X}}}}(x,y)$ 
the polynomial $\varphi_{\ell}^{G_{\rm{X}}}(x,y)$ divided by 
its $x^\ell$ coefficient.
We give some examples in Table \ref{Tab:sum2}. 
\begin{table}[thb]
\caption{Examples of Eisenstein polynomials}
\label{Tab:sum2}
\begin{center}
\begin{tabular}{c|c}
\noalign{\hrule height0.8pt}
$\ell$ & $\widetilde{\varphi_{\ell}^{G_{\rm{II}}}}(x,y)$\\\hline
$8$ & $x^8+14 x^4 y^4+y^8$\\\hline
$12$ & $x^{12}-33 x^8 y^4-33 x^4 y^8+y^{12}$\\
\noalign{\hrule height0.8pt}
\end{tabular}
\end{center}
\end{table}






In \cite{Ouratalk,Oura3}, several analogies between 
Eisenstein series and 
Eisenstein polynomials were reported. 
Suppose $p$ is a prime number and $v_p$ is the corresponding 
valuation for the field $\QQ$. 
Then $a\in \QQ$ is said to be $p$-integral if $v_p(a)\geq 0$. 
Eisenstein series have the following properties: 
\begin{enumerate}
\item [(1)]
All of the zeros of the Eisenstein series are 
on the circle 
$\{e^{\sqrt{-1}\theta}\mid \pi/2\leq \theta \leq 2\pi/3\}$ 
\cite{RS}. 
\item [(2)]
The zeros of the Eisenstein series $\psi_k^{\Gamma_1}(\tau)$ are the same as those for $\psi_{k+2}^{\Gamma_1}(\tau)$ 
\cite{Nozaki}. 
\item [(3)]
For odd prime $p$, where $p\geq 5$, 
the coefficients of the Eisenstein series $\psi_{p-1}^{\Gamma_1}(\tau)$ 
are $p$-integral 
\cite[P.~233, Theorem 3]{IR}, \cite{Kob}. 
\end{enumerate}
Oura's conjecture states that 
the analogous properties of (1), (2), and (3) also hold for 
$Th(\widetilde{\varphi_\ell^{G_{\rm{II}}}})$, given formally as the following conjecture. 

\begin{conj}[\cite{Ouratalk,Oura3}]\label{conj;Oura}
\begin{enumerate}
\item [{\rm (1)}]
All of the zeros of $Th(\widetilde{\varphi_\ell^{G_{\rm{II}}}})$ are 
on a segment of a circle
$\{e^{\sqrt{-1}\theta}\mid \pi/2\leq \theta \leq 2\pi/3\}$. 
\item [{\rm (2)}]
The zeros of $Th(\widetilde{\varphi_\ell^{G_{\rm{II}}}})$ are the same as 
those of $Th(\widetilde{\varphi_{\ell+8}^{G_{\rm{II}}}})$. 
\item [{\rm (3)}]
Let $p$ be an odd prime. 
Then the coefficients of $Th(\widetilde{\varphi_{2(p-1)}^{G_{\rm{II}}}})$ 
are $p$-integral. 
\end{enumerate}
\end{conj}

To explain our results, 
we introduce the zeta polynomials, 
which were defined by Duursma \cite{D1}.
Analogous to coding theory, we say $f \in \CC[x, y]$ 
is the formal weight enumerator of degree $n$ 
if $f$ is a homogeneous polynomial of degree $n$ 
and the coefficient of $x^n$ is one. 
Also, for
\[
f (x, y) = x^n +
\sum_{i=d}^n
A_ix^{n-i}y^i\ (A_d \neq 0),
\]
$d$ is the minimum distance of $f$. 
Let $R$ be a commutative ring and $R[[T]]$ be the formal power series 
ring over $R$. 
For
$Z(T) =
\sum_{i=0}^{\infty}a_nT_n \in R[[T]]$, 
$[T^k]Z(T)$ denotes the coefficient $a_k$. 
This gives the following lemma. 
\begin{lem}[cf.~\cite{D1}]\label{lem:D}
Let $f$ be a formal weight enumerator of degree $n$, 
$d$ be the minimum distance, and $q$ be any real number not one.
Then there exists a unique polynomial $P_f(T) \in \CC[T]$ 
of degree at most $n-d$ such that the following equation holds: 
\[
[T^{n-d}]
\frac{P_f(T)}
{(1 - T )(1 -qT ) }
(xT + y(1 - T ))^n =
\frac{f (x, y) - x^n}
{q - 1}.
\]
\end{lem}

\begin{df}[cf.~\cite{D2}]
For a formal weight enumerator $f$, we call the polynomial 
$P_f(T)$ determined in Lemma \ref{lem:D} the zeta polynomial of $f$ with respect to $q$. 
If all the zeros of $P_f(T)$ have absolute value $1/\sqrt{q}$, 
then we say that $f$ satisfies the Riemann hypothesis analogues (RHA). 
\end{df}

In \cite{miezaki}, 
we investigated the zeta polynomials of 
the Eisenstein polynomials for Type II. 
In the following, we assume that $q=2$. 
The cases of $\ell = 8$ and $\ell = 12$ are listed in Table \ref{Tab:sum3}. 
\begin{table}[thb]
\caption{Examples of zeta polynomials}
\label{Tab:sum3}
\begin{center}
\begin{tabular}{c|c}
\noalign{\hrule height0.8pt}
$\ell$ & $P_{\widetilde{\varphi_{\ell}^{G_{\rm{II}}}}}(T)$\\\hline
$8$ & $\frac{1}{5}+\frac{2 T}{5}+\frac{2 T^2}{5}$\\\hline
$12$ & $-\frac{1}{15}-\frac{2 T}{15}-\frac{2 T^2}{15}+\frac{4 T^4}{15}+\frac{8 T^5}{15}+\frac{8 T^6}{15}$\\
\noalign{\hrule height0.8pt}
\end{tabular}
\end{center}
\end{table}






In the previous paper, 
it was shown that 
Oura's observation for the zeta polynomial associated 
with Eisenstein polynomials holds. 
\begin{thm}[\cite{miezaki}]\label{thm:main1}
\begin{itemize}
\item [{\rm (I)}]
\begin{enumerate}
\item [{\rm (1)}]
$P_{\widetilde{\varphi_{\ell}^{G_{\rm{II}}}}}(T)$ satisfies RHA. 

\item [{\rm (2)}]
The zeros of $P_{\widetilde{\varphi_{\ell}^{G_{\rm{II}}}}}(T)$ interlace  those of $P_{\widetilde{\varphi_{\ell+8}^{G_{\rm{II}}}}}(T)$. 

\item [{\rm (3)}]
Let $p$ be an odd prime with $p\neq 5$. Then 
the coefficients of $P_{\widetilde{\varphi_{2(p-1)}^{G_{\rm{II}}}}}(T)$ are $p$-integral. 

\end{enumerate}
\item [{\rm (II)}]
Let $p$ be an odd prime. Then 
the coefficients of $\widetilde{\varphi_{2(p-1)}^{G_{\rm{II}}}}(x,y)$ are $p$-integral. 

\item [{\rm (III)}]
Conjecture \ref{conj;Oura} (3) is true. 
\end{itemize}
\end{thm}

The main purpose of the present paper 
is to show that similar results hold for the remaining cases. 
We set the values of $w_{\rm{X}}$ and $q_{\rm{X}}$ the same as in the previous work for consistency:\begin{tabbing}
Type I: $w_{\rm{I}}=2$, $q_{\rm{I}}=2$, \\
Type II: $w_{\rm{II}}=8$, $q_{\rm{II}}=2$, \\
Type III: $w_{\rm{III}}=3$, $q_{\rm{III}}=3$, \\
Type IV: $w_{\rm{IV}}=2$, $q_{\rm{IV}}=4$. 
\end{tabbing}
\begin{thm}\label{thm:main}
For $\rm{X}\in \{\rm{I},\rm{III},\rm{IV}\}$, 
\begin{itemize}
\item [{\rm (I)}]
\begin{enumerate}
\item [{\rm (1)}]
$P_{\widetilde{\varphi_{\ell}^{G_{\rm{X}}}}}(T)$ satisfies RHA for $q_{\rm{X}}$. 

\item [{\rm (2)}]
The zeros of $P_{\widetilde{\varphi_{\ell}^{G_{\rm{X}}}}}(T)$ interlace  those of $P_{\widetilde{\varphi_{\ell+w_{\rm{X}}}^{G_{\rm{X}}}}}(T)$. 

\item [{\rm (3)}]
Let $p$ be an odd prime and assume that $p\neq 3$ for the Type III case. Then 
the coefficients of $P_{\widetilde{\varphi_{2(p-1)}^{G_{\rm{X}}}}}(T)$ are $p$-integral. 

\end{enumerate}
\item [{\rm (II)}]
Let $p$ be an odd prime and assume that $p\neq 3$ for the Type III case. Then 
the coefficients of $\widetilde{\varphi_{2(p-1)}^{G_{\rm{X}}}}(x,y)$ are $p$-integral. 

\item [{\rm (III)}]
Let $p$ be an odd prime and assume that $p\neq 3$ for the Type III case. 
The coefficients of $Th(\widetilde{\varphi_{2(p-1)}^{G_{\rm{X}}}})$ 
are $p$-integral. 
\end{itemize}
\end{thm}

In Section $2$, 
the proof of Theorem \ref{thm:main} is provided along with concluding remarks. 

\section{Proof of Theorem \ref{thm:main}}

In this section, 
we provide the proof of Theorem \ref{thm:main}. 

\subsection{Preliminaries}

Before proving Theorem \ref{thm:main}, 
we first recall a property of zeta polynomials. 

The zeta polynomial $P_f(T)$ associated with $f$ is related 
to the normalized weight enumerator of $f$ as follows: \begin{df}[cf.~\cite{D3}]\label{df:nwe}
For a formal weight enumerator $f(x,y)=\sum_{i=0}^{n} A_ix^{n-i}y^i$, 
we make the following definition of a normalized weight enumerator.
\begin{align*}
N_f(t)=\frac{1}{q-1}\sum_{i=d}^{n} A_i/\binom{n}{i}t^{i-d}. 
\end{align*}
\end{df}
The relation between $P_f(T)$ and $N_f(t)$ is given by the following theorem. 
\begin{thm}[cf.~\cite{D3}]\label{thm:D3}
For a given formal weight enumerator $f(x,y)$ with minimum distance $d$, 
the zeta polynomial $P_f(T)$ and the normalized 
weight enumerator $N_f(t)$ have the following relation: 
\[
\frac{P_f(T)}{(1-T)(1-qT)}(1-T)^{d+1}\equiv N_f\left(\frac{T}{1-T}\right)\pmod{T^{n-d+1}}. 
\]
\end{thm}

\subsection{Explicit forms of Eisenstein polynomials}

The explicit forms of the Eisenstein polynomials 
$\varphi_\ell^{G_{\rm{X}}}(x,y)$ are given by the following theorem.
\begin{thm}\label{thm:T}
\begin{enumerate}
\item [{\rm (1)}]

{\rm Type I:} 
\begin{align*}
&\widetilde{\varphi_\ell^{G_{\rm{I}}}}(x,y)\\
&=
\left\{
\begin{array}{ll}
\displaystyle
x^\ell + y^{\ell} +\frac{2}{2+\sqrt{2}^{\ell}}
\sum_{0<j<\ell,j\equiv 0\pmod{2}}
\binom{\ell}{j} x^{\ell-j}y^j &\mbox{ if } \ell\equiv 0\pmod{2},\\
0 &\mbox{ if } \ell\not\equiv 0\pmod{2}. 
\end{array}
\right.
\end{align*}

\item [{\rm (2)}]

{ \rm Type III:} 
\begin{align*}
&\widetilde{\varphi_\ell^{G_{\rm {III}}}}(x,y)\\
&=
\left\{
\begin{array}{ll}
\displaystyle
x^\ell + \frac{3}{3+\sqrt{3}^{\ell}}
\sum_{0<j<\ell,j\equiv 0\pmod{3}}
2^{j} \binom{\ell}{j} x^{\ell-j}y^j 
&\mbox{ if } \ell\equiv 0\pmod{4},\\
0 &\mbox{ if } \ell\not\equiv 0\pmod{4}. 
\end{array}
\right.
\end{align*}

\item [{\rm (3)}]

{\rm Type IV:} 
\begin{align*}
&\widetilde{\varphi_\ell^{G_{\rm{IV}}}}(x,y)\\
&=
\left\{
\begin{array}{ll}
\displaystyle
x^\ell + \frac{2}{2+2^{\ell}}
\sum_{0<j<\ell,j\equiv 0\pmod{2}}
3^{j} \binom{\ell}{j} x^{\ell-j}y^j
&\mbox{ if } \ell\equiv 0\pmod{2},\\
0 &\mbox{ if } \ell\not\equiv 0\pmod{2}. 
\end{array}
\right.
\end{align*}

\end{enumerate}

\end{thm}
\begin{proof}
We prove only the Type I case. 
In Appendix A, 
we give proofs covering the other cases.

By a direct calculation, 
\begin{align*}
\widetilde{\varphi_\ell^{G_{\rm{I}}}}&(x,y)\\
&=
\frac{1}{2}\frac{1}{1+2(\frac{1}{\sqrt{2}})^\ell}
\left(
x^\ell
+\left(\frac{1}{\sqrt{2}}(x+y)\right)^\ell
+\left(\frac{1}{\sqrt{2}}(x-y)\right)^\ell
+y^\ell\right.\\
&
\left.+
(-x)^\ell
+\left(-\frac{1}{\sqrt{2}}(x+y)\right)^\ell
+\left(-\frac{1}{\sqrt{2}}(x-y)\right)^\ell
+(-y)^\ell
\right). 
\end{align*}
Then the result follows. 

The elements of $G_{\rm{I}}$ are 
listed on the homepage of one of the authors \cite{miezakiweb}. 

\end{proof}

\subsection{Explicit forms of zeta polynomials}
To prove Theorem \ref{thm:main}, 
we provide explicit formulas for the zeta function associated with the
Eisenstein polynomials $\widetilde{\varphi_{\ell}^{G_{\rm{X}}}}$ in the following theorem.
\begin{thm}\label{thm:Z}
\begin{enumerate}
\item [{\rm (1)}]

{\rm Type I:} For $\ell\equiv 0\pmod{2}$, 
\[
P_{\widetilde{\varphi_{\ell}^{G_{\rm{I}}}}}(T)
=\frac{2+\sqrt{2}^{\ell}T^{\ell-2}}{2+\sqrt{2}^\ell}. 
\]

\item [{\rm (2)}]

{ \rm Type III:}  For $\ell\equiv 0\pmod{4}$, 
\[
P_{\widetilde{\varphi_{\ell}^{G_{\rm{III}}}}}(T)=
\frac{3\cdot 2^2}{3+\sqrt{3}^{\ell}}
\sum_{j=0}^{(\ell-4)/2}
(-3)^jT^{2j}. 
\]

\item [{\rm (3)}]

{\rm Type IV:}  For $\ell\equiv 0\pmod{2}$, 
\[
P_{\widetilde{\varphi_{\ell}^{G_{\rm{IV}}}}}(T)=
\frac{6}{2+2^\ell}
\sum_{j=0}^{\ell-2}
(-2)^jT^{j}. 
\]

\end{enumerate}

\end{thm}

\begin{proof}

We prove only the Type I case. 
In Appendix B, 
we give proofs covering the other cases.

Let 
$N_{\widetilde{\varphi_{\ell}^{G_{\rm{I}}}}}$ be the normalized weight enumerator of 
$\varphi_{\ell}^{G_{\rm{I}}}$. 
By Definition \ref{df:nwe}, 
we have 
\begin{align*}
N_{\widetilde{\varphi_\ell^{G_{\rm{I}}}}}(t)
&=
\sum_{0<j<\ell,j\equiv 0\pmod{2}}\frac{2}{2+\sqrt{2}^\ell}
t^{j - 2}+t^{\ell- 2}\\
&\equiv 
\frac{2}{2+\sqrt{2}^\ell}\frac{1}{1-t^2}
+\left(1-\frac{2}{2+\sqrt{2}^\ell}\right)t^{\ell- 2}\pmod{t^{\ell-1}}. 
\end{align*}
Then, by Theorem \ref{thm:D3}, 
we have 
\begin{align*}
&\frac{P_{\widetilde{\varphi_\ell^{G_{\rm{I}}}}}(T)}{(1-T)(1-2T)}(1-T)^{3}\equiv 
N_{\widetilde{\varphi_\ell^{G_{\rm{I}}}}}\left(\frac{T}{1-T}\right)\pmod{T^{\ell-1}}\\
\Leftrightarrow &
P_{\widetilde{\varphi_\ell^{G_{\rm{I}}}}}(T)
\equiv 
N_{\widetilde{\varphi_\ell^{G_{\rm{I}}}}}\left(\frac{T}{1-T}\right)
\frac{(1-T)(1-2T)}{(1-T)^{3}} \pmod{T^{\ell-1}}\\
&
\equiv 
\frac{2}{2+\sqrt{2}^\ell}\frac{(1-T)^2}{(1-T)^2-T^2}
\frac{(1-T)(1-2T)}{(1-T)^{3}}\\
&+\left(1-\frac{2}{2+\sqrt{2}^\ell}\right)\left(\frac{T}{1-T}\right)^{\ell- 2}
\frac{(1-T)(1-2T)}{(1-T)^{3}} \pmod{T^{\ell-1}}\\
&
\equiv 
\frac{2+\sqrt{2}^\ell T^{\ell-2}}{2+\sqrt{2}^\ell}
\pmod{T^{\ell-1}}. 
\end{align*}
From this, we have 
$$
P_{\widetilde{\varphi_{\ell}^{G_{\rm{I}}}}}(T)=\frac{2+\sqrt{2}^\ell T^{\ell-2}}{2+\sqrt{2}^\ell}. 
$$

\end{proof}

\subsection{Proof of Theorem \ref{thm:main}}
In this section, we will present the proof of Theorem \ref{thm:main} after that of the following lemma. 
\begin{lem}\label{lem:mod}
Let $\ell=2(p-1)$ for some odd prime. 
\begin{enumerate}
\item [{\rm (1)}]

\[
2+\sqrt{2}^\ell\not\equiv 0\pmod{p}. 
\]

\item [{\rm (2)}]

If $p\neq 3$, then  
\[
3+\sqrt{3}^{\ell}\not\equiv 0\pmod{p}. 
\]

\item [{\rm (3)}]

\[
2+2\ell\not\equiv 0\pmod{p}. 
\]

\end{enumerate}
\end{lem}
\begin{proof}
We prove only (1). 
The other assertions can be proved similarly. 


By Fermat's little theorem, 
\[
2+\sqrt{2}^\ell=2+2^{p-1}\equiv 2 \not\equiv 0\pmod{p}. 
\]




\end{proof}

\begin{proof}[Proof of Theorem \ref{thm:main}]

Here, again, we prove only the Type I case. 
The other cases can be proved similarly. 

Clearly, (I)--(1) and (I)--(2) follow from Theorem \ref{thm:Z}. 

For (I)--(3), 
we recall that 
\[
P_{\widetilde{\varphi_{\ell}^{G_{\rm{I}}}}}(T)
=\frac{2+\sqrt{2}^{\ell}T^{\ell-2}}{2+\sqrt{2}^\ell}. 
\]

Then, Theorem \ref{thm:main} (I)--(3) follows from Lemma \ref{lem:mod} (1). 

To show (II), we first recall that 
\[
\widetilde{\varphi_\ell^{G_{\rm{I}}}}(x,y)
=
x^\ell + y^{\ell} +\frac{2}{2+\sqrt{2}^{\ell}}
\sum_{0<j<\ell,j\equiv 0\pmod{2}}
\binom{\ell}{j} x^{\ell-j}y^j. 
\]
By Lemma \ref{lem:mod} (1), 
the coefficients of $\widetilde{\varphi_{2(p-1)}^{G_{\rm{I}}}}(x,y)$ 
are $p$-integral.

Finally, we show (III). 
By Theorem \ref{thm:main} (II), 
the coefficients of 
\[
\widetilde{\varphi_\ell^{G_{\rm{I}}}}(x,y)
=
x^\ell + y^{\ell} +\frac{2}{2+\sqrt{2}^{\ell}}
\sum_{0<j<\ell,j\equiv 0\pmod{2}}
\binom{\ell}{j} x^{\ell-j}y^j
\]
are $p$-integral. The theta maps $f_0$ and $f_1$ have integral Fourier coefficients. This completes the proof. 
\end{proof}

\subsection{Concluding Remarks}
\begin{rem}
\begin{enumerate}
\item [(1)]
The definition of the Eisenstein polynomial for genus $g$ is given in \cite{{Oura1},{Oura2}}. 
In the present paper, 
we only consider the genus one ($g=1$) case. 
For the cases with $g>1$, 
do the analogies still hold? 

\item [(2)]
For $\rm{X}\in \{\rm{I},\ldots,\rm{IV}\}$, 
the group $G_{\rm{X}}$ is a finite unitary reflection group. 
These groups are classified in \cite{ST}, 
which gives rise to a natural question: 
For the other unitary reflection groups, 
do our analogies still hold? 

\end{enumerate}
\end{rem}

\section*{Acknowledgments}
The authors thank Koji Chinen and Iwan Duursma for their helpful discussions 
and contributions to this research. 
The authors would also like to thank the anonymous
reviewers for their beneficial comments on an earlier version of the manuscript. The authors are supported by JSPS KAKENHI (18K03217,17K05164).

\appendix 
\section{Proof of Theorem \ref{thm:T} for the cases Type III and Type IV}

\begin{proof}[Proof of Theorem \ref{thm:T} (2)]

By a direct calculation, 
\begin{align*}
\widetilde{\varphi_\ell^{G_{\rm{III}}}}&(x,y)\\
&=
\frac{1}{4}\frac{1}{1+3(\frac{1}{\sqrt{3}})^\ell}\\
&\left(
x^\ell
+\left(\frac{1}{\sqrt{3}}(x+2y)\right)^\ell
\right.\\
&
\left.
+\left(\frac{1}{\sqrt{3}}(x+\sqrt{-3}y-y)\right)^\ell
+\left(\frac{1}{\sqrt{3}}(\sqrt{-1}x+\sqrt{3}y-\sqrt{-1}y)\right)^\ell
\right.\\
&
+
(-x)^\ell
+\left(-\frac{1}{\sqrt{3}}(x+2y)\right)^\ell\\
&\left.
+\left(-\frac{1}{\sqrt{3}}(x+\sqrt{-3}y-y)\right)^\ell
+\left(-\frac{1}{\sqrt{3}}(\sqrt{-1}x+\sqrt{3}y-\sqrt{-1}y)\right)^\ell
\right.\\
&
+
(\sqrt{-1}x)^\ell
+\left(\frac{\sqrt{-1} (x+2 y)}{\sqrt{3}}\right)^\ell\\
&
\left.+\left(\frac{1}{3} \sqrt{-1} \left(\sqrt{3} x-\left(\sqrt{3}-3 \sqrt{-1}\right) y\right)
\right)^\ell
+\left(-\frac{x}{\sqrt{3}}+\sqrt{-1} y+\frac{y}{\sqrt{3}}\right)^\ell
\right.\\
&
+
(-\sqrt{-1}x)^\ell
+\left(-\frac{\sqrt{-1} (x+2 y)}{\sqrt{3}}\right)^\ell\\
&
\left.
+\left(\frac{x}{\sqrt{3}}-\frac{1}{3} \left(\sqrt{3}+3 \sqrt{-1}\right) y\right)^\ell
+\left(-\frac{\sqrt{-1} x}{\sqrt{3}}+y+\frac{\sqrt{-1} y}{\sqrt{3}}\right)^\ell
\right). 
\end{align*}
Then the result follows. 

The elements of $G_{\rm{III}}$ are 
listed on the homepage of one of the authors \cite{miezakiweb}. 

\end{proof}

\begin{proof}[Proof of Theorem \ref{thm:T} (3)]

By a direct calculation, 
\begin{align*}
\widetilde{\varphi_\ell^{G_{\rm{IV}}}}&(x,y)\\
&=
\frac{1}{2}\frac{1}{1+2(\frac{1}{2})^\ell}
\left(
x^\ell
+\left(\frac{1}{2}(x+3y)\right)^\ell
+\left(\frac{1}{2}(x-3y)\right)^\ell
\right.\\
&
\left.+
(-x)^\ell
+\left(-\frac{1}{2}(x+3y)\right)^\ell
+\left(-\frac{1}{2}(x-3y)\right)^\ell
\right). 
\end{align*}
Then the result follows. 

The elements of $G_{\rm{IV}}$ are 
listed on the homepage of one of the authors \cite{miezakiweb}. 
\end{proof}

\section{Proof of Theorem \ref{thm:Z} for the cases Type III and Type IV}

\begin{proof}[Proof of Theorem \ref{thm:Z} (2)]
Let 
$N_{\widetilde{\varphi_{\ell}^{G_{\rm{III}}}}}$ be the normalized weight enumerator of 
$\varphi_{\ell}^{G_{\rm{III}}}$. 
By Definition \ref{df:nwe}, 
we have 
\begin{align*}
N_{\widetilde{\varphi_\ell^{G_{\rm{III}}}}}(t)
&=
\frac{1}{2}\frac{3}{3+3^{\ell/2}}
\sum_{0<j<\ell,j\equiv 0\pmod{3}}2^{j}t^{j-3}\\
&\equiv 
\frac{3\cdot 2^2}{3+3^{\ell/2}}\frac{1}{1-(2T)^3}\pmod{t^{\ell-2}}. 
\end{align*}
Then, by Theorem \ref{thm:D3}, 
we have 
\begin{align*}
&\frac{P_{\widetilde{\varphi_\ell^{G_{\rm{III}}}}}(T)}{(1-T)(1-3T)}(1-T)^{4}\equiv 
N_{\widetilde{\varphi_\ell^{G_{\rm{III}}}}}\left(\frac{T}{1-T}\right)\pmod{T^{\ell-2}}\\
\Leftrightarrow &
P_{\widetilde{\varphi_\ell^{G_{\rm{III}}}}}(T)
\equiv 
N_{\widetilde{\varphi_\ell^{G_{\rm{III}}}}}\left(\frac{T}{1-T}\right)
\frac{(1-T)(1-3T)}{(1-T)^{4}} \pmod{T^{\ell-2}}\\
&
\equiv 
\frac{3\cdot 2^2}{3+3^{\ell/2}}\frac{(1-T)^3}{(1-T)^3-(2T)^3}
\frac{(1-T)(1-3T)}{(1-T)^{4}} \pmod{T^{\ell-2}}\\
&
\equiv 
\frac{3\cdot 2^2}{3+3^{\ell/2}}
\frac{1}{1+3T^2}
\pmod{T^{\ell-2}}. 
\end{align*}
From this, we have 
$$
P_{\widetilde{\varphi_{\ell}^{G_{\rm{III}}}}}(T)=
\frac{3\cdot 2^2}{3+3^{\ell/2}}
\sum_{j=0}^{(\ell-4)/2}
(-3)^jT^{2j}. 
$$

\end{proof}

\begin{proof}[Proof of Theorem \ref{thm:Z} (3)]
Let 
$N_{\widetilde{\varphi_{\ell}^{G_{\rm{IV}}}}}$ be the normalized weight enumerator of 
$\varphi_{\ell}^{G_{\rm{IV}}}$. 
By Definition \ref{df:nwe}, 
we have 
\begin{align*}
N_{\widetilde{\varphi_\ell^{G_{\rm{IV}}}}}(t)
&=
\frac{1}{3}\frac{2}{2+2^{\ell}}
\sum_{0<j<\ell,j\equiv 0\pmod{2}}3^{j}t^{j-2}\\
&\equiv 
\frac{2\cdot 3}{2+2^{\ell}}\frac{1}{1-(3T)^2}\pmod{t^{\ell-1}}. 
\end{align*}
Then, by Theorem \ref{thm:D3}, 
we have 
\begin{align*}
&\frac{P_{\widetilde{\varphi_\ell^{G_{\rm{IV}}}}}(T)}{(1-T)(1-4T)}(1-T)^{3}\equiv 
N_{\widetilde{\varphi_\ell^{G_{\rm{IV}}}}}\left(\frac{T}{1-T}\right)\pmod{T^{\ell-1}}\\
\Leftrightarrow &
P_{\widetilde{\varphi_\ell^{G_{\rm{IV}}}}}(T)
\equiv 
N_{\widetilde{\varphi_\ell^{G_{\rm{IV}}}}}\left(\frac{T}{1-T}\right)
\frac{(1-T)(1-4T)}{(1-T)^{3}} \pmod{T^{\ell-1}}\\
&
\equiv 
\frac{1}{3}\frac{2}{2+2^{\ell}}\frac{(1-T)^2}{(1-T)^2-(3T)^2}
\frac{(1-T)(1-4T)}{(1-T)^{3}} \pmod{T^{\ell-1}}\\
&
\equiv 
\frac{6}{2+2\ell}
\frac{1}{1+2T}
\pmod{T^{\ell-1}}. 
\end{align*}
From this, we have 
$$
P_{\widetilde{\varphi_{\ell}^{G_{\rm{IV}}}}}(T)=
\frac{6}{2+2\ell}
\sum_{j=0}^{\ell-2}
(-2)^jT^{j}. 
$$

\end{proof}

\end{document}